\documentclass[10pt,reqno]{amsproc}

\usepackage{amsmath, amsthm, amscd, amsfonts, amssymb, graphicx, color}
\usepackage[bookmarksnumbered, plainpages]{hyperref}
\usepackage{graphicx}
\usepackage{epsfig}

\usepackage{graphicx}
\usepackage{epsfig}
\newtheorem{theorem}{Theorem}[section]
\newtheorem{thm}[theorem]{\bf{Theorem}}
\newtheorem{cor}[theorem]{\bf{Corollary}}
\newtheorem{lem}[theorem]{\bf{Lemma}}
\newtheorem{prop}[theorem]{\bf{Proposition}}

\newtheorem{remark}[theorem]{\bf{Remark}}
\newtheorem{ex}[theorem]{\bf{Example}}

\numberwithin{equation}{section}

 \newcommand{\A}{\mathcal{A}}
   \newcommand{\B}{\mathcal{B}}
 \newcommand{\C}{\mathcal{C}}
   \newcommand{\T}{\mathcal{T}}
 \newcommand{\M}{\mathcal{M}}
 
 \newcommand{\N}{\mathcal{N}}

 \newcommand{\Pa}{\mathcal{P}}

\begin{document}
%------------------------------------------------------------------------------------%

%------------------------------------------------------------------------------------%
\title[Derivations on Triangular Banach Algebras of Order Three]{Derivations on Triangular Banach Algebras of Order Three}
\author{Ali Ebadian$^1$}
\address{
 $^1$Department of Mathematics, Payame Noor University  \newline
\indent P.O. BOX 19395-3697, Tehran, Iran}
\email{ebadian.ali@gmail.com}
\author{Madjid Eshaghi Gordji$^2$}
\address{$^2$Department of Mathematics, Semnan University\newline
\indent P. O. Box 35195-363, Semnan, Iran}
\address{$^2$Center of Excellence in Nonlinear Analysis and Applications (CENAA)\newline
\indent Semnan University, Iran}
\email{madjid.eshaghi@gmail.com}
\author{Ali Jabbari$^3$}
\address{
 $^3$Young Researchers Club, Ardabil Branch  \newline
\indent Islamic Azad University,  Ardabil, Iran}
\email{jabbari\underline{ }al@yahoo.com}

 \subjclass[2000]{Primary: 05C50, Secondary: 47B47}

\keywords{Banach algebra, Derivation, Matrix algebra, Triangular algebra}
\maketitle
\begin{abstract}
In this paper, we define some new notions of triangular Banach algebras and we investigate the  derivations on these algebras.
\end{abstract}

%------------------------------------------------------------------------------------%

%------------------------------------------------------------------------------------%

\section{ Introduction}
Let $\A$, $\B$ and $\C$ be  algebras. Consider the triangular matrix of order three
$$\T=\left[
    \begin{array}{ccc}
      \A& \M & \Pa \\
       & \B & \N \\
       &  & \C \\
    \end{array}
  \right],
$$
where $\M$ is left $\A$-module, right $\B$-module, $\N$ is left $\B$-module, right $\C$-module, and $\Pa$ is left $\A$-module, and right $\C$-module. The matrix $\T$ is become an algebra by usual adding and product of $3\times3$ matrix, via
\begin{equation*}
   \left[
  \begin{array}{ccc}
    a_1 & m_1 & p_1 \\
     & b_1 & n_1 \\
     &  & c_1 \\
  \end{array}
\right]+ \left[
  \begin{array}{ccc}
    a_2 & m_2 & p_2 \\
     & b_2 & n_2 \\
     &  & c_2 \\
  \end{array}
\right]= \left[
  \begin{array}{ccc}
    a_1+a_2 & m_1+m_2 & p_1+p_2 \\
     & b_1 +b_2& n_1+n_2 \\
     &  & c_1+c_2 \\
  \end{array}
\right],
\end{equation*}
and
\begin{eqnarray*}
   \left[
  \begin{array}{ccc}
    a_1 & m_1 & p_1 \\
     & b_1 & n_1 \\
     &  & c_1 \\
  \end{array}
\right] \left[
  \begin{array}{ccc}
    a_2 & m_2 & p_2 \\
     & b_2 & n_2 \\
     &  & c_2 \\
  \end{array}
\right] &&  \\
  & \hspace{-4cm}=&   \hspace{-2cm}\left[
  \begin{array}{ccc}
    a_1a_2 & a_1m_2+m_1b_2 & a_1p_2+\mu(m_1\otimes n_2)+p_1 c_2\\
     & b_1 b_2& b_1n_2+n_1c_2 \\
     &  & c_1c_2 \\
  \end{array}
\right],
\end{eqnarray*}
where $\mu:\M\otimes \N\longrightarrow\Pa$ is a left $\A$-module and right $\C$-module homomorphism. Now, suppose $\A$, $\B$ and $\C$ are Banach algebras, and $\M$ is left Banach $\A$-module, right Banach $\B$-module, $\N$ is left Banach $\B$-module, right Banach $\C$-module, and $\Pa$ is left Banach  $\A$-module, and right Banach $\C$-module. Then the triangular algebra $\T$ is a Banach algebra  by the following norm
\begin{equation*}
    \| \left[
  \begin{array}{ccc}
    a & m & p \\
     & b & n \\
     &  & c \\
  \end{array}
\right]\|=\|a\|_\A+\|m\|_\M+\|p\|_{\Pa}+\|b\|_\B+\|n\|_\N+\|c\|_C.
\end{equation*}

We  identify this algebra  by $\T=\A\oplus_1\M\oplus_1\Pa\oplus_1\B\oplus_1\N\oplus_1\C$.

Recently, some results regarding  homomorphisms and generalized homomorphisms on order three ring matrixes obtained by Xing in \cite{xi}. This paper motivated us that we study the derivations on these rings. Derivations of $2\times2$ triangular matrix algebras  studied by Forrest and Marcoux in \cite{fo1}. In this paper, by using the ideas of  Forrest and Marcoux, we study the derivations on triangular matrix algebras.

\section{Main Results}
We start our work  with the following easy but essential Proposition which play important role in characterizations of derivations on triangular matrix algebras of order three.
\begin{prop}\label{p1}
Let $\A$, $\B$ and $\C$ be unital Banach algebras, and let $D:\T\longrightarrow\T$ be a derivation. Then there exist derivations $D_\A:\A\longrightarrow\A$, $D_\B:\B\longrightarrow\B$, $D_\C:\C\longrightarrow\C$, the linear mappings $\tau_\M:\M\longrightarrow\M$, $\tau_{\Pa}:\Pa\longrightarrow\Pa$, $\tau_\N:\N\longrightarrow\N$, and elements $m_D\in\M$, $p_D\in\Pa$ and $n_D\in\N$ such that the following statements hold:
\begin{enumerate}
  \item $D\Big{(}\left[
  \begin{array}{ccc}
    e_\A & 0 & 0 \\
     & 0 & 0 \\
     &  & 0 \\
  \end{array}
\right]\Big{)}=\left[
  \begin{array}{ccc}
    0 & m_D & p_D \\
     & 0 & 0 \\
     &  & 0 \\
  \end{array}
\right]$.
  \item $D\Big{(}\left[
  \begin{array}{ccc}
    a & 0 & 0 \\
     & 0 & 0 \\
     &  & 0 \\
  \end{array}
\right]\Big{)}=\left[
  \begin{array}{ccc}
    D_\A(a) & am_D & ap_D \\
     & 0 & 0 \\
     &  & 0 \\
  \end{array}
\right]$.
  \item $D\Big{(}\left[
  \begin{array}{ccc}
    0 & 0 & 0 \\
     & e_\B & 0 \\
     &  & 0 \\
  \end{array}
\right]\Big{)}=\left[
  \begin{array}{ccc}
    0 & -m_D & 0 \\
     & 0 & n_D \\
     &  & 0 \\
  \end{array}
\right]$.
  \item $D\Big{(}\left[
  \begin{array}{ccc}
    0 & 0 & 0 \\
     & b & 0 \\
     &  & 0 \\
  \end{array}
\right]\Big{)}=\left[
  \begin{array}{ccc}
    0 & -m_Db & 0 \\
     & D_\B(b) & bn_D \\
     &  & 0 \\
  \end{array}
\right]$.
  \item $D\Big{(}\left[
  \begin{array}{ccc}
    0 & 0 & 0 \\
     & 0 & 0 \\
     &  & e_\C \\
  \end{array}
\right]\Big{)}=\left[
  \begin{array}{ccc}
    0 & 0 & -p_D \\
     & 0 & -n_D \\
     &  & 0 \\
  \end{array}
\right]$.
  \item $D\Big{(}\left[
  \begin{array}{ccc}
    0 & 0 & 0 \\
     & 0 & 0 \\
     &  & c \\
  \end{array}
\right]\Big{)}=\left[
  \begin{array}{ccc}
    0 & 0 & -p_Dc \\
     & 0 & -n_Dc \\
     &  & D_\C(c) \\
  \end{array}
\right]$.
\item $D\Big{(}\left[
  \begin{array}{ccc}
    0 & m & 0 \\
     & 0 & 0 \\
     &  & 0 \\
  \end{array}
\right]\Big{)}=\left[
  \begin{array}{ccc}
    0 & \tau_\M(m) & \mu(m\otimes n_D) \\
     & 0 & 0 \\
     &  & 0 \\
  \end{array}
\right]$.
\item $D\Big{(}\left[
  \begin{array}{ccc}
    0 & 0 & p \\
     & 0 & 0 \\
     &  & 0 \\
  \end{array}
\right]\Big{)}=\left[
  \begin{array}{ccc}
    0 & 0 & \tau_{\Pa}(p) \\
     & 0 & 0 \\
     &  & 0 \\
  \end{array}
\right]$.
\item $D\Big{(}\left[
  \begin{array}{ccc}
    0 & 0 & 0 \\
     & 0 & n \\
     &  & 0\\
  \end{array}
\right]\Big{)}=\left[
  \begin{array}{ccc}
    0 & 0 & \mu(-m_D\otimes n) \\
     & 0 & \tau_\N(n) \\
     &  & 0\\
  \end{array}
\right]$.
\end{enumerate}
\end{prop}
\begin{proof}
We prove just cases (1), (2) and (7), and  other cases have similar proof. Let $D\Big{(}\left[
  \begin{array}{ccc}
    e_\A & 0 & 0 \\
     & 0 & 0 \\
     &  & 0 \\
  \end{array}
\right]\Big{)}=\left[
  \begin{array}{ccc}
    \alpha & m' & p' \\
     &\beta & n' \\
     &  & \gamma \\
  \end{array}
\right]$. Since $D$ is a derivation, then we have
\begin{eqnarray*}
  D\Big{(}\left[
  \begin{array}{ccc}
    e_\A & 0 & 0 \\
     & 0 & 0 \\
     &  & 0 \\
  \end{array}
\right]\Big{)} &=& \left[
  \begin{array}{ccc}
    e_\A & 0 & 0 \\
     & 0 & 0 \\
     &  & 0 \\
  \end{array}
\right]\left[
  \begin{array}{ccc}
    \alpha & m' & p' \\
     &\beta & n' \\
     &  & \gamma \\
  \end{array}
\right]+\left[
  \begin{array}{ccc}
    \alpha & m' & p' \\
     &\beta & n' \\
     &  & \gamma \\
  \end{array}
\right] \left[
  \begin{array}{ccc}
    e_\A & 0 & 0 \\
     & 0 & 0 \\
     &  & 0 \\
  \end{array}
\right]\\
   &=& \left[
  \begin{array}{ccc}
    0& m' & p' \\
     &0& 0 \\
     &  & 0 \\
  \end{array}
\right]=\left[
  \begin{array}{ccc}
    0 & m_D & p_D \\
     & 0 & 0 \\
     &  & 0 \\
  \end{array}
\right].
\end{eqnarray*}

For (2), existence of $D_\A$ by easy calculation is clear, assume that $D\Big{(}\left[
  \begin{array}{ccc}
    a & 0 & 0 \\
     & 0 & 0 \\
     &  & 0 \\
  \end{array}
\right]\Big{)}=\left[
  \begin{array}{ccc}
    D_\A(a) & m' & p' \\
     &\beta & n' \\
     &  & \gamma \\
  \end{array}
\right]$. Then by (1) we have
\begin{eqnarray*}
  D\Big{(}\left[
  \begin{array}{ccc}
    a & 0 & 0 \\
     & 0 & 0 \\
     &  & 0 \\
  \end{array}
\right]\Big{)} &=&\left[
  \begin{array}{ccc}
    a & 0 & 0 \\
     & 0 & 0 \\
     &  & 0 \\
  \end{array}
\right] \left[
  \begin{array}{ccc}
    0 & m_D & p_D \\
     &0 & 0 \\
     &  & 0 \\
  \end{array}
\right] +\left[
  \begin{array}{ccc}
    D_\A(a) & m' & p' \\
     &\beta & n' \\
     &  & \gamma \\
  \end{array}
\right]\left[
  \begin{array}{ccc}
    e_\A & 0 & 0 \\
     & 0 & 0 \\
     &  & 0 \\
  \end{array}
\right]\\
   &=&\left[
  \begin{array}{ccc}
    D_\A(a) & am_D & ap_D \\
     & 0 & 0 \\
     &  & 0 \\
  \end{array}
\right].
\end{eqnarray*}

(7) Suppose that $D\Big{(}\left[
  \begin{array}{ccc}
    0 & m & 0 \\
     & 0 & 0 \\
     &  & 0 \\
  \end{array}
\right]\Big{)}=\left[
  \begin{array}{ccc}
    \alpha & \tau_\M(m) & p' \\
     & \beta & n' \\
     &  & \gamma \\
  \end{array}
\right]$. Then (1) and (3) imply
\begin{eqnarray}\label{1}
\nonumber
  D\Big{(}\left[
  \begin{array}{ccc}
    0 & m & 0 \\
     & 0 & 0 \\
     &  & 0 \\
  \end{array}
\right]\Big{)} &=&\left[
  \begin{array}{ccc}
    e_\A & 0 & 0 \\
     & 0 & 0 \\
     &  & 0 \\
  \end{array}
\right] \left[
  \begin{array}{ccc}
    \alpha & \tau_\M(m) & p' \\
     & \beta & n' \\
     &  & \gamma \\
  \end{array}
\right] +\left[
  \begin{array}{ccc}
    0 & m_D & p_D \\
     &0 & 0 \\
     &  & 0 \\
  \end{array}
\right] \left[
  \begin{array}{ccc}
    0 & m & 0 \\
     & 0 & 0 \\
     &  & 0 \\
  \end{array}
\right]\\
   &=&\left[
  \begin{array}{ccc}
    \alpha & \tau_\M(m) & p' \\
     & 0& 0 \\
     &  & 0 \\
  \end{array}
\right],
\end{eqnarray}
and
\begin{eqnarray}\label{2}
\nonumber
  D\Big{(}\left[
  \begin{array}{ccc}
    0 & m & 0 \\
     & 0 & 0 \\
     &  & 0 \\
  \end{array}
\right]\Big{)} &=& \left[
  \begin{array}{ccc}
    0 & m & 0 \\
     & 0 & 0 \\
     &  & 0 \\
  \end{array}
\right]\left[
  \begin{array}{ccc}
    0 & -m_D & \\
     &0 & n_D \\
     &  & 0 \\
  \end{array}
\right]+ \left[
  \begin{array}{ccc}
    \alpha & \tau_\M(m) & p' \\
     & \beta & n' \\
     &  & \gamma \\
  \end{array}
\right]\left[
  \begin{array}{ccc}
    0 & 0 & 0 \\
     & e_\B & 0 \\
     &  & 0 \\
  \end{array}
\right] \\
   &=&\left[
  \begin{array}{ccc}
    0 & \tau_\M(m) & \mu(m\otimes n_D) \\
     & 0& 0 \\
     &  & 0 \\
  \end{array}
\right],
\end{eqnarray}
thus, by (\ref{1}) and (\ref{2}), we desert  the result.
\end{proof}

By collecting of obtained results in Proposition \ref{p1}, we have the following.
\begin{cor}\label{c1}
Let $D:\T\longrightarrow\T$ be a derivation, then there exist derivations $D_\A:\A\longrightarrow\A$, $D_\B:\B\longrightarrow\B$, $D_\C:\C\longrightarrow\C$, the linear mappings $\tau_\M:\M\longrightarrow\M$, $\tau_{\Pa}:\Pa\longrightarrow\Pa$, $\tau_\N:\N\longrightarrow\N$, and elements $m_D\in\M$, $p_D\in\Pa$ and $n_D\in\N$ such that
\begin{equation*}
  D\Big{(}\left[
  \begin{array}{ccc}
   a & m & p \\
     & b & n \\
     &  & c \\
  \end{array}
\right]\Big{)}=\left[
  \begin{array}{ccc}
    D_\A(a) & am_D-m_Db+\tau_\M(m) & ap_D-p_Dc+\mu(-m_D\otimes n)\\
    &&+\mu(m\otimes n_D)+\tau_{\Pa}(p) \\
     & D_\B(b) &bn_D-n_Dc+\tau_\N(n) \\
     &  & D_\C(c) \\
  \end{array}
\right],
\end{equation*}
for any $\left[
  \begin{array}{ccc}
   a & m & p \\
     & b & n \\
     &  & c \\
  \end{array}
\right]\in\T$.
\end{cor}

We introduced the mappings $\tau_\M$, $\tau_{\Pa}$ and $\tau_\N$ on $\M$, $\Pa$ and $\N$, respectively. Now, we have the following lemma.
\begin{lem}\label{l1}
Let $D:\T\longrightarrow\T$ be a derivation, then  the following statements hold
\begin{enumerate}
  \item $\tau_\M(am)=D_\A(a)m+a\tau_\M(m)$.
  \item $\tau_\M(mb)=mD_\B(b)+\tau_\M(m)b$.
  \item $\tau_{\Pa}(ap)=D_\A(a)p+a\tau_{\Pa}(p)$.
  \item $\tau_{\Pa}(pc)=pD_\C(c)+\tau_{\Pa}(p)c$.
  \item $\tau_\N(bn)=D_\B(b)n+b\tau_\N(n)$.
  \item $\tau_\N(nc)=nD_\C(c)+\tau_\N(n)c$.
\end{enumerate}
\end{lem}
\begin{proof}
We only prove (1), other cases are similar. By Proposition \ref{p1} and Corollary \ref{c1}, we have
\begin{eqnarray*}
   D\Big{(}\left[
  \begin{array}{ccc}
   0 & am & 0\\
     &  0& 0 \\
     &  & 0 \\
  \end{array}
\right]\Big{)} &=&D\Big{(}\left[
  \begin{array}{ccc}
   a & 0 & 0\\
     &  0& 0 \\
     &  & 0 \\
  \end{array}
\right]\left[
  \begin{array}{ccc}
   0 & m & 0\\
     &  0& 0 \\
     &  & 0 \\
  \end{array}
\right]\Big{)} =\left[
  \begin{array}{ccc}
   a & 0 & 0\\
     &  0& 0 \\
     &  & 0 \\
  \end{array}
\right] \left[
  \begin{array}{ccc}
   0 & \tau_\M(m) & 0\\
     &  0& 0 \\
     &  & 0 \\
  \end{array}
\right] \\
   && +\left[
  \begin{array}{ccc}
    D_\A(a) & am_D & ap_D \\
     & 0 & 0 \\
     &  & 0 \\
  \end{array}
\right]\left[
  \begin{array}{ccc}
   0 & m & 0\\
     &  0& 0 \\
     &  & 0 \\
  \end{array}
\right]\\
   &=&  \left[
  \begin{array}{ccc}
   0 &  D_\A(a)m+a\tau_\M(m) & 0 \\
     & 0 & 0 \\
     &  & 0 \\
  \end{array}
\right]=  \left[
  \begin{array}{ccc}
   0 &  \tau_\M(am) & 0 \\
     & 0 & 0 \\
     &  & 0 \\
  \end{array}
\right].
\end{eqnarray*}
\end{proof}

In the following theorem, we do not need to assume that $\A$, $\B$ and $\C$ are unital, and our proof is algebraic.
\begin{thm}\label{t1}
Let  $D_\A:\A\longrightarrow\A$, $D_\B:\B\longrightarrow\B$, and $D_\C:\C\longrightarrow\C$, be continuous derivations and  $\tau_\M:\M\longrightarrow\M$, $\tau_{\Pa}:\Pa\longrightarrow\Pa$, and $\tau_\N:\N\longrightarrow\N$ be continuous linear mappings. If the linear mappings $\tau_\M, \tau_{\Pa}$, and $\tau_\N$ satisfy in cases (1)-(6) of the Lemma \ref{l1},   and
$$\tau_{\Pa}(\mu(m\otimes n))=\mu(\tau_\M(m)\otimes n)+\mu(m\otimes \tau_\N(n)),$$
then $D:\T\longrightarrow\T$ defined by
\begin{equation*}
  D\Big{(}\left[
  \begin{array}{ccc}
   a & m & p \\
     & b & n \\
     &  & c \\
  \end{array}
\right]\Big{)}=\left[
  \begin{array}{ccc}
    D_\A(a) & \tau_\M(m) & \tau_{\Pa}(p) \\
     & D_\B(b) &\tau_\N(n) \\
     &  & D_\C(c) \\
  \end{array}
\right],
\end{equation*}
for any $\left[
  \begin{array}{ccc}
   a & m & p \\
     & b & n \\
     &  & c \\
  \end{array}
\right]\in\T$, is a continuous derivation.
\end{thm}
\begin{proof}
Continuity of $D$ by it's definition is clear.
\begin{eqnarray*}
D\Big{(}\left[
  \begin{array}{ccc}
    a_1 & m_1 & p_1 \\
     & b_1 & n_1 \\
     &  & c_1 \\
  \end{array}
\right] \left[
  \begin{array}{ccc}
    a_2 & m_2 & p_2 \\
     & b_2 & n_2 \\
     &  & c_2 \\
  \end{array}
\right]\Big{)} &&   \\
   &\hspace{-8cm}=&\hspace{-4cm} D\Big{(}\left[
  \begin{array}{ccc}
    a_1a_2 & a_1m_2+m_1b_2 & a_1p_2+\mu(m_1\otimes n_2)+p_1 c_2\\
     & b_1 b_2& b_1n_2+n_1c_2 \\
     &  & c_1c_2 \\
  \end{array}
\right] \Big{)}\\
    &\hspace{-8cm}=&\hspace{-4cm} \left[
  \begin{array}{ccc}
    D_\A(a_1a_2) & \tau_\M( a_1m_2+m_1b_2) & \tau_{\Pa}(a_1p_2+\mu(m_1\otimes n_2)+p_1 c_2)\\
     & D_\B(b_1 b_2)& \tau_\N(b_1n_2+n_1c_2) \\
     &  & D_\C(c_1c_2) \\
  \end{array}
\right].
\end{eqnarray*}

Conversely,
\begin{eqnarray*}
   &&  \left[
  \begin{array}{ccc}
    a_1 & m_1 & p_1 \\
     & b_1 & n_1 \\
     &  & c_1 \\
  \end{array}
\right]\left[
  \begin{array}{ccc}
    D_\A(a_2) & \tau_\M(m_2) & \tau_{\Pa}(p_2) \\
     & D_\B(b_2) &\tau_\N(n_2) \\
     &  & D_\C(c_2) \\
  \end{array}
\right]\\
   && \hspace{-0.35cm}+\left[
  \begin{array}{ccc}
    D_\A(a_1) & \tau_\M(m_1) & \tau_{\Pa}(p_1) \\
     & D_\B(b_1) &\tau_\N(n_1) \\
     &  & D_\C(c_1) \\
  \end{array}
\right]\left[
  \begin{array}{ccc}
    a_2 & m_2 & p_2 \\
     & b_2 & n_2 \\
     &  & c_2 \\
  \end{array}
\right] \\
   &=& \left[
  \begin{array}{ccc}
   a_1 D_\A(a_2) & a_1\tau_\M(m_2)+m_1D_\B(b_2) & a_1\tau_{\Pa}(p_2)+\mu(m_1\otimes \tau_\N(n_2))+p_1D_\C(c_2) \\
     & b_1D_\B(b_2) &b_1\tau_\N(n_2)+n_1D_\C(c_2) \\
     &  & c_1D_\C(c_2) \\
  \end{array}
\right] \\
   && \hspace{-0.35cm}+\left[
  \begin{array}{ccc}
    D_\A(a_1)a_2 & D_\A(a_1)m_2+\tau_\M(m_1)b_2 & D_\A(a_1)p_2+\mu(\tau_\M(m_1)\otimes n_2)+\tau_{\Pa}(p_1)c_2 \\
     & D_\B(b_1)b_2 &D_\B(b_1)n_2+\tau_\N(n_1)c_2 \\
     &  &D_\C(c_1) c_2\\
  \end{array}
\right].
\end{eqnarray*}

Therefore $D$ is a derivation.
\end{proof}

Note that if $\mu(\M\otimes N)=0$, then the mapping $D:\T\longrightarrow\T$  defined in the above theorem is a derivation.

We denote the space of all continuous left $\A$-module morphisms and right $\B$-module morphisms on $\M$ by $\emph{\emph{Hom}}_{\A,\B}(\M)$, if $\A=\B$, we write $\emph{\emph{Hom}}_{\A}(\M)$. Similarly we define $\emph{\emph{Hom}}_{\A,\C}(\Pa)$ and $\emph{\emph{Hom}}_{\B,\C}(\N)$.

Consider the continuous mappings $\tau_\M:\M\longrightarrow\M$, $\tau_{\Pa}:\Pa\longrightarrow\Pa$ and $\tau_\N:\N\longrightarrow\N$. We say these maps are generalized Rosenblum operators, if there exist derivations $D_\A:\A\longrightarrow\A$, $D_\B:\B\longrightarrow\B$ and $D_\C:\C\longrightarrow\C$ such that
\begin{equation*}
    \tau_\M(amb)=D_\A(a)mb+a\tau_\M(m)b+amD_\B(b),
\end{equation*}
\begin{equation*}
    \tau_{\Pa}(apc)=D_\A(a)pc+a\tau_{\Pa}(p)c+apD_\C(c),
\end{equation*}
and
\begin{equation*}
    \tau_\N(bnc)=D_\B(b)nc+b\tau_\N(n)c+bnD_\C(c),
\end{equation*}
for every $a\in\A, b\in\B, c\in\C, m\in\M, p\in\Pa$ and $c\in\C$. We denote the generalized Rosenblum operator on $\M$ specified by $x\in\A$ and $y\in\B$, by $\tau_\M^{x,y}$, and similarly, we denote the generalized Rosenblum operators on $\Pa$ and $\N$ specified by $x\in\A$, $y\in\B$, and $z\in\C$ by $\tau_{\Pa}^{x,z}$ and $\tau_\N^{y,z}$, respectively.

By $\mathcal{Z}(\A),\mathcal{Z}(\B)$ and $\mathcal{Z}(\C)$, we mean the center of the Banach algebras $\A$, $\B$ and $\C$, respectively. Let $x\in\mathcal{Z}(\A),y\in\mathcal{Z}(\B)$ and $z\in\mathcal{Z}(\C)$, the operators $\tau_\M^{x,y}$, $\tau_{\Pa}^{x,z}$ and $\tau_\N^{y,z}$ called central Rosenblum operators on $\M$, $\Pa$ and $\N$. These operators defined as follows
\begin{equation*}
    \tau_\M^{x,y}(m)=my-xm,~\tau_{\Pa}^{x,z}(p)=pz-xp,~\tau_{\N}^{y,z}(n)=nz-yn.
\end{equation*}

The space of all central Rosenblum operators on $\M$, $\Pa$ and $\N$ denoted by $\mathcal{Z}\mathcal{R}_{\A,\B}(\M)$, $\mathcal{ZR}_{\A,\C}(\Pa)$ and $\mathcal{ZR}_{\B,\C}(\N)$, respectively.

\begin{lem}\label{l2}
Let $\T$ be a triangular Banach algebra of order three defined as above. Then
\begin{itemize}
  \item[(i)]  $\mathcal{Z}\mathcal{R}_{\A,\B}(\M)\subseteq{\emph{Hom}}_{\A,\B}(\M)$.
  \item[(ii)] $\mathcal{ZR}_{\A,\C}(\Pa)\subseteq{\emph{Hom}}_{\A,\C}(\Pa)$.
  \item[(iii)] $\mathcal{ZR}_{\B,\C}(\N)\subseteq {\emph{Hom}}_{\B,\C}(\N)$.
\end{itemize}
\end{lem}
\begin{proof}
It follows by a same reasoning as   proof of Lemma 2.6 of \cite{fo1}.

\end{proof}
\begin{lem}\label{l3}
Let $\varphi\in{\emph{Hom}}_{\A,\B}(\M)$, $\theta\in{\emph{Hom}}_{\A,\C}(\Pa)$ and $\psi\in{\emph{Hom}}_{\B,\C}(\N)$ such that
\begin{equation}\label{e1}
    \theta(\mu(m\otimes n))=\mu(\varphi(m)\otimes n)+\mu(m\otimes\psi(n)).
\end{equation}

Then $D_{\varphi,\theta,\psi}:\T\longrightarrow\T$ defined by $D_{\varphi,\theta,\psi}\Big{(}\left[
  \begin{array}{ccc}
   a & m & p \\
     & b & n \\
     &  & c \\
  \end{array}
\right]\Big{)}=\left[
  \begin{array}{ccc}
    0 & \varphi(m) & \theta(p) \\
     & 0 &\psi(n) \\
     &  & 0 \\
  \end{array}
\right]$ is a continuous derivation. Moreover, $D_{\varphi,\theta,\psi}$ is an  inner derivation if and only if there exist $x\in\A$, $y\in\B$ and $z\in C$ such that $\varphi=\tau_\M^{x,y}\in\mathcal{Z}\mathcal{R}_{\A,\B}(\M)$, $\theta=\tau_{\Pa}^{x,z}\in\mathcal{ZR}_{\A,\C}(\Pa)$ and $\psi=\tau_\N^{y,z}\in\mathcal{ZR}_{\B,\C}(\N)$.
\end{lem}
\begin{proof}
By Theorem \ref{t1}, $D_{\varphi,\theta,\psi}$ is a continuous derivation. Now, suppose that $D_{\varphi,\theta,\psi}$ is inner. Therefore there exists $\left[\begin{array}{ccc}
   x & \alpha & \beta \\
     & y & \gamma \\
     &  & z \\
  \end{array}
\right]\in\T$ such that
\begin{eqnarray*}
  D_{\varphi,\theta,\psi}\Big{(}\left[
  \begin{array}{ccc}
   a & m & p \\
     & b & n \\
     &  & c \\
  \end{array}
\right]\Big{)} &=&  \left[
  \begin{array}{ccc}
   a & m & p \\
     & b & n \\
     &  & c \\
  \end{array}
\right]\left[\begin{array}{ccc}
   x & \alpha & \beta \\
     & y & \gamma \\
     &  & z \\
  \end{array}
\right]-\left[\begin{array}{ccc}
   x & \alpha & \beta \\
     & y & \gamma \\
     &  & z \\
  \end{array}
\right]\left[
  \begin{array}{ccc}
   a & m & p \\
     & b & n \\
     &  & c \\
  \end{array}
\right]\\
   &=& \left[
  \begin{array}{ccc}
   ax & a\alpha+my & a\beta+\mu(m\otimes\gamma)+pz \\
     & by & b\gamma+nz \\
     &  & cz \\
  \end{array}
\right] \\
   && \hspace{-0.4cm} - \left[
  \begin{array}{ccc}
   xa & xm+\alpha b & xp+\mu(\alpha\otimes n)+\beta c \\
     & yb & yn+\gamma c \\
     &  & zc \\
  \end{array}
\right]\\
&=&\left[
  \begin{array}{ccc}
   ax-xa & a\alpha-\alpha b +my -xm& a\beta-\beta c+pz-xp \\
   && +\mu(m\otimes\gamma)-\mu(\alpha\otimes n)\\
     & by-yb & b\gamma-\gamma c+nz-yn \\
     &  & cz-zc \\
  \end{array}
\right]
\end{eqnarray*}

On the other hand $D_{\varphi,\theta,\psi}\Big{(}\left[
  \begin{array}{ccc}
   a & m & p \\
     & b & n \\
     &  & c \\
  \end{array}
\right]\Big{)}=\left[
  \begin{array}{ccc}
    0 & \varphi(m) & \theta(p) \\
     & 0 &\psi(n) \\
     &  & 0 \\
  \end{array}
\right]$. Thus, $ax-xa=0$, $by-yb=0$ and $cz-zc=0$. It follows that $x\in\mathcal{Z}(\A)$, $y\in\mathcal{\B}$ and $z\in\mathcal{Z}(C)$. Moreover, we have
\begin{enumerate}
  \item $\varphi(m)=a\alpha-\alpha b +my -xm$.
  \item $\theta(p)=a\beta-\beta c+pz-xp+\mu(m\otimes\gamma)-\mu(\alpha\otimes n)$.
  \item $\psi(n)=b\gamma-\gamma c+nz-yn$.
\end{enumerate}

Since $\varphi\in\emph{{\emph{Hom}}}_{\A,\B}(\M)$, $\theta\in\emph{{\emph{Hom}}}_{\A,\C}(\Pa)$ and $\psi\in\emph{{\emph{Hom}}}_{\B,\C}(\N)$, therefore we conclude that in (1) $a\alpha-\alpha b=0$, and $\varphi(m)=my -xm=\tau_\M^{x,y}(m)$; in (2) $\theta(p)=a\beta-\beta c+\mu(m\otimes\gamma)-\mu(\alpha\otimes n)=0$ and $\theta(p)=pz-xp=\tau_{\Pa}^{x,z}(p)$, and in (3) $b\gamma-\gamma c=0$ and $\psi(n)=nz-yn=\tau_\N^{y,z}(n)$.

Conversely, assume $\varphi=\tau_\M^{x,y}\in\mathcal{Z}\mathcal{R}_{\A,\B}(\M)$, $\theta=\tau_{\Pa}^{x,z}\in\mathcal{ZR}_{\A,\C}(\Pa)$ and $\psi=\tau_\N^{y,z}\in\mathcal{ZR}_{\B,\C}(\N)$. Define the inner derivation $D:\T\longrightarrow\T$ specified by $\left[\begin{array}{ccc}
   x & 0 & 0 \\
     & y & 0 \\
     &  & z \\
  \end{array}
\right]\in\T$. Then we have
\begin{eqnarray*}
   D_{\left[\begin{array}{ccc}
   x & 0 & 0 \\
     & y & 0 \\
     &  & z \\
  \end{array}
\right]}\Big{(}\left[
  \begin{array}{ccc}
   a & m & p \\
     & b & n \\
     &  & c \\
  \end{array}
\right]\Big{)}&=& \left[
  \begin{array}{ccc}
   a & m & p \\
     & b & n \\
     &  & c \\
  \end{array}
\right]\left[\begin{array}{ccc}
   x & 0 & 0 \\
     & y & 0 \\
     &  & z \\
  \end{array}
\right]-\left[\begin{array}{ccc}
   x & 0 & 0 \\
     & y & 0 \\
     &  & z \\
  \end{array}
\right]\left[
  \begin{array}{ccc}
   a & m & p \\
     & b & n \\
     &  & c \\
  \end{array}
\right] \\
   &=& \left[\begin{array}{ccc}
   ax & my& pz \\
     & by & nz \\
     &  & cz \\
  \end{array}
\right]-\left[\begin{array}{ccc}
   xa & xm & xp \\
     & yb & yn \\
     &  & zc \\
  \end{array}
\right] \\
   &=& \left[
  \begin{array}{ccc}
    0 & \varphi(m) & \theta(p) \\
     & 0 &\psi(n) \\
     &  & 0 \\
  \end{array}
\right]= D_{\varphi,\theta,\psi}\Big{(}\left[
  \begin{array}{ccc}
   a & m & p \\
     & b & n \\
     &  & c \\
  \end{array}
\right]\Big{)}.
\end{eqnarray*}

This means that $D_{\varphi,\theta,\psi}$ is inner.
\end{proof}

Note that in the last part of above proof, we do not use condition (\ref{e1}). Now, we are ready to prove the main theorem of this paper.
\begin{thm}\label{t2}
Let $\T$ be a triangular Banach algebra of order three. If $\mathcal{H}^1(\A)=0$, $\mathcal{H}^1(\B)=0$ and $\mathcal{H}^1(\C)=0$, then
\begin{equation*}
    \mathcal{H}^1(\T)\cong \frac{{{\emph{Hom}}}_{\A,\B}(\M)\oplus_1{{\emph{Hom}}}_{\A,\C}(\Pa)\oplus_1{{\emph{Hom}}}_{\B,\C}(\N)}{\mathcal{Z}\mathcal{R}_{\A,\B}(\M) \oplus_1\mathcal{ZR}_{\A,\C}(\Pa)\oplus_1\mathcal{ZR}_{\B,\C}(\N)}.
\end{equation*}
\end{thm}
\begin{proof}
Define $\Phi:\emph{\emph{Hom}}_{\A,\B}(\M)\oplus_1\emph{\emph{Hom}}_{\A,\C}(\Pa)\oplus_1\emph{\emph{Hom}}_{\B,\C}(\N)\longrightarrow \mathcal{H}^1(\T)$ by $\Phi(\varphi,\theta,\psi)=\overline{D}_{\varphi,\theta,\psi}$, where $\overline{D}_{\varphi,\theta,\psi}$ is the equivalence class of ${D}_{\varphi,\theta,\psi}$ in $\mathcal{H}^1(\T)$. It easy to show that $\Phi$ is linear. We have  to show that $\Phi$ is onto. To this end, let $D:\T\longrightarrow\T$ be a continuous derivation.  Then the Corollary \ref{c1} implies that there exist derivations $D_\A:\A\longrightarrow\A$, $D_\B:\B\longrightarrow\B$, $D_\C:\C\longrightarrow\C$, the linear mappings $\tau_\M:\M\longrightarrow\M$, $\tau_{\Pa}:\Pa\longrightarrow\Pa$, $\tau_\N:\N\longrightarrow\N$, and elements $m_D\in\M$, $p_D\in\Pa$ and $n_D\in\N$ such that
\begin{equation*}
  D\Big{(}\left[
  \begin{array}{ccc}
   a & m & p \\
     & b & n \\
     &  & c \\
  \end{array}
\right]\Big{)}=\left[
  \begin{array}{ccc}
    D_\A(a) & am_D-m_Db+\tau_\M(m) & ap_D-p_Dc+\mu(-m_D\otimes n)\\
    &&+\mu(m\otimes n_D)+\tau_{\Pa}(p) \\
     & D_\B(b) &bn_D-n_Dc+\tau_\N(n) \\
     &  & D_\C(c) \\
  \end{array}
\right],
\end{equation*}
and since $\mathcal{H}^1(\A)$, $\mathcal{H}^1(\B)$ and $\mathcal{H}^1(\C)$ are zero, so there are $x\in\A$, $y\in\B$ and $z\in\C$ such that $D_\A(a)=ax-xa=D_x(a)$, $D_\B(b)=by-yb=D_y(b)$ and $D_\C(c)=cz-zc=D_z(c)$. Define $D_0:\T\longrightarrow\T$ as follows
\begin{equation*}
  D_0\Big{(}\left[
  \begin{array}{ccc}
   a & m & p \\
     & b & n \\
     &  & c \\
  \end{array}
\right]\Big{)}=\left[
  \begin{array}{ccc}
    D_x(a) & am_D-m_Db+\tau_\M^{x,y}(m) & ap_D-p_Dc+\mu(-m_D\otimes n)\\
    &&+\mu(m\otimes n_D)+\tau_{\Pa}^{x,z}(p) \\
     & D_y(b) &bn_D-n_Dc+\tau_\N^{y,z}(n) \\
     &  & D_z(c) \\
  \end{array}
\right].
\end{equation*}

For every $\left[
  \begin{array}{ccc}
   a & m & p \\
     & b & n \\
     &  & c \\
  \end{array}
\right]\in\T$, we have
\begin{eqnarray*}
   && \left[
  \begin{array}{ccc}
   a & m & p \\
     & b & n \\
     &  & c \\
  \end{array}
\right]\left[
  \begin{array}{ccc}
   x & m_D & p_D \\
     & y & n_D \\
     &  & z \\
  \end{array}
\right]-\left[
  \begin{array}{ccc}
   x & m_D & p_D \\
     & y & n_D \\
     &  & z \\
  \end{array}
\right]\left[
  \begin{array}{ccc}
   a & m & p \\
     & b & n \\
     &  & c \\
  \end{array}
\right] \\
   &=&  \left[
  \begin{array}{ccc}
   ax-xa & am_D-m_Db+my-xm & ap_D-p_Dc+\mu(m\otimes n_D)+\mu(-m_D\otimes n)+pz-xp \\
     & by-yb & bn_D-n_Dc+nz-yn \\
     &  & cz-zc \\
  \end{array}
\right].
\end{eqnarray*}

It follows that $D_0$ is a inner derivation specified by $\left[
  \begin{array}{ccc}
   x & m_D & p_D \\
     & y & n_D \\
     &  & z \\
  \end{array}
\right]$. Define $D_1=D-D_0$. Then $D_1$ is a derivation and we have
\begin{eqnarray*}
  D_1\Big{(}\left[
  \begin{array}{ccc}
   a & m & p \\
     & b & n \\
     &  & c \\
  \end{array}
\right]\Big{)} &=& \left[
  \begin{array}{ccc}
    0 & \tau_\M(m)-\tau_\M^{x,y}(m) & \tau_{\Pa}(p)-\tau_{\Pa}^{x,z}(p) \\
     & 0 &\tau_\N(n)-\tau_\N^{y,z}(n) \\
     &  & 0 \\
  \end{array}
\right] \\
   &=& \left[
  \begin{array}{ccc}
    0 & \tau'_\M(m) & \tau'_{\Pa}(p) \\
     & 0 &\tau'_\N(n) \\
     &  & 0 \\
  \end{array}
\right],
\end{eqnarray*}
where $\tau_\M-\tau_\M^{x,y}=\tau'_\M$, $\tau_{\Pa}-\tau_{\Pa}^{x,z}=\tau'_{\Pa}$ and $\tau_\N-\tau_\N^{y,z}=\tau'_\N$. Clearly, by Lemma \ref{l2} and properties of $\tau_\M$, $\tau_{\Pa}$ and $\tau_\N$, we have $\tau'_\M\in\emph{\emph{Hom}}_{\A,\B}(\M)$, $\tau'_{\Pa}\in\emph{\emph{Hom}}_{\A,\C}(\Pa)$ and $\tau'_\N\in\emph{\emph{Hom}}_{\B,\C}(\N)$. Hence, $\overline{D}=\overline{D}_1=\Phi(\tau'_\M,\tau'_{\Pa},\tau'_\N)$. It means that $\Phi$ is onto. Therefore
\begin{equation*}
    \mathcal{H}^1(\T)\cong \frac{{{\emph{Hom}}}_{\A,\B}(\M)\oplus_1{{\emph{Hom}}}_{\A,\C}(\Pa)\oplus_1{{\emph{Hom}}}_{\B,\C}(\N)}{\ker\Phi}.
\end{equation*}

Suppose $(\varphi,\theta,\psi)\in\ker\Phi$. Then by Lemma \ref{l3},  $(\varphi,\theta,\psi)\in\ker\Phi$ if and only if $D_{\varphi,\theta,\psi}$ is inner. This shows that $\ker\Phi=\mathcal{Z}\mathcal{R}_{\A,\B}(\M) \oplus_1\mathcal{ZR}_{\A,\C}(\Pa)\oplus_1\mathcal{ZR}_{\B,\C}(\N)$.
\end{proof}
\begin{remark}\label{r1}
Note that the above obtained results  are true when we suppose that $\B$ and $\C$ have bounded approximate identity \emph{(}i.e. they are non-unital\emph{)}. Therefore by this notation, we can write all results for unital Banach algebra $\A$ and the non-untal Banach algebras $\B$ and $\C$ with bounded approximate identity.
\end{remark}
\begin{ex}
Suppose that $\A=\B=\C=\mathbb{C}$, \emph{(}$\mathbb{C}$ is the space of complex number\emph{)}, and $\M$, $\Pa$ and $\N$ are arbitrary Banach spaces. Then, it is known that every derivation from $\mathbb{C}$ on itself is inner. Thus, $\mathcal{H}^1(\A)$, $\mathcal{H}^1(\B)$ and $\mathcal{H}^1(\C)$ are zero. By $\mathcal{L}(\M)$, we mean the space of all linear and bounded operator from $\M$ into $\M$ \emph{(}$\mathcal{L}(\Pa)$ and $\mathcal{L}(\N)$ have similar definition\emph{)}. Then by an  easy argument, we can conclude that ${\emph{Hom}}_{\mathbb{C}}(\M)=\mathcal{L}(\M)$ and $\mathcal{Z}\mathcal{R}_{\mathbb{C}}(\M)=\{\lambda \emph{id}_\M:\lambda\in\mathbb{C}\}$, where $\emph{id}_\M$ is the identity map on $\M$. Similarly, one can  obtain same results for $\Pa$ and $\N$. Then by Theorem \ref{t2}, we have
\begin{equation*}
    \mathcal{H}^1(\T)\cong \frac{\mathcal{L}(\M)\oplus_1\mathcal{L}(\Pa)\oplus_1\mathcal{L}(\N)}{\{\lambda \emph{id}_\M:\lambda\in\mathbb{C}\}\oplus_1\{\lambda \emph{id}_{\Pa}:\lambda\in\mathbb{C}\}\oplus_1\{\lambda \emph{id}_{\N}:\lambda\in\mathbb{C}\}}.
\end{equation*}
\end{ex}
\begin{ex}
Let $G$, $H$ and $K$ be locally compact groups, and $\A=M(G)$, $\B=M(H)$ and $\C=M(K)$ be measure algebras on$G$, $H$ and $K$, respectively. By Theorem 5.6.34 \emph{(iii)} of \cite{da}, $$\mathcal{H}^1(M(G))=\mathcal{H}^1(L^1(G),M(G))$$ and by Corollary 1.4 of \cite{lo}, we have $$\mathcal{H}^1(L^1(G),M(G))=0.$$  By the same reasoning  one can obtain that $\mathcal{H}^1(M(H))=0$ and $\mathcal{H}^1(M(K))=0$. Let $G_1=G\cap H\neq\emptyset$, $G_2=G\cap K\neq\emptyset$ and $G_3=H\cap K\neq\emptyset$. Now, consider the triangular Banach algebra
$$\T=\left[
  \begin{array}{ccc}
   M(G) & L^1(G_1) & L^1(G_2) \\
     & M(H) &L^1(G_3) \\
     &  & M(K) \\
  \end{array}
\right].$$

 Then by Theorem \ref{t2}, we have
 \begin{equation*}
    \mathcal{H}^1(\T)\cong \frac{{{\emph{Hom}}}_{M(G),M(H)}(L^1(G_1))\oplus_1{{\emph{Hom}}}_{M(G),M(K)}(L^1(G_3))\oplus_1{{\emph{Hom}}}_{M(H),M(K)}(L^1(G_2))}{\mathcal{Z}\mathcal{R}_{M(G),M(H)}(L^1(G_1)) \oplus_1\mathcal{ZR}_{M(G),M(K)}(L^1(G_3))\oplus_1\mathcal{ZR}_{M(H),M(K)}(L^1(G_2))}.
\end{equation*}
\end{ex}
\begin{ex}
Let $G$ be a locally compact group and $M(G)$ and $L^1(G)$ be  group algebras on $G$. Consider
$$\T=\left[
  \begin{array}{ccc}
   M(G) & L^1(G) & L^1(G) \\
     & M(G) &L^1(G) \\
     &  & M(G) \\
  \end{array}
\right].$$

Then
\begin{equation*}
    \mathcal{H}^1(\T)\cong \frac{{{\emph{Hom}}}_{M(G)}(L^1(G))\oplus_1{{\emph{Hom}}}_{M(G)}(L^1(G))\oplus_1{{\emph{Hom}}}_{M(G)}(L^1(G))}{\mathcal{Z}\mathcal{R}_{M(G)}(L^1(G)) \oplus_1\mathcal{ZR}_{M(G)}(L^1(G))\oplus_1\mathcal{ZR}_{M(G)}(L^1(G))}.
\end{equation*}
\end{ex}
\begin{ex}
Let $G$ be abelian locally compact group, and
$$\T=\left[
  \begin{array}{ccc}
   \ell^1(G) & L^1(G) & L^1(G) \\
     & L^1(G) &L^1(G) \\
     &  & L^1(G) \\
  \end{array}
\right].$$

Johnson's theorem implies that $\ell^1(G)$ and $L^1(G)$ are commutative amenable Banach algebras, so $$\mathcal{H}^1(\ell^1(G))=0=\mathcal{H}^1(L^1(G)).$$  Wendel's theorem \emph{(}\cite[Theorem 1]{we}\emph{)} implies that
\begin{equation*}
    {{\emph{Hom}}}_{\ell^1(G),L^1(G)}(L^1(G))=M(G)~\emph{and}~ {{\emph{Hom}}}_{L^1(G),L^1(G)}(L^1(G))=M(G).
\end{equation*}

As well as,
\begin{equation*}
    \mathcal{Z}\mathcal{R}_{\ell^1(G),L^1(G)}(L^1(G))=\ell^1(G)\oplus L^1(G) ~\emph{and}~\mathcal{ZR}_{L^1(G),L^1(G)}(L^1(G))=L^1(G)\oplus L^1(G).
\end{equation*}

Therefore by Theorem \ref{t2} and Remark \ref{r1}, we have
\begin{equation*}
    \mathcal{H}^1(\T)\cong \frac{M(G)\oplus_1M(G)\oplus_1M(G)}{(\ell^1(G)\oplus L^1(G)) \oplus_1(\ell^1(G)\oplus L^1(G))\oplus_1(L^1(G)\oplus L^1(G))}.
\end{equation*}
\end{ex}
%------------------------------------------------------------------------------------%

%-----------------------------------------------------------------------------

\end{document}